\documentclass[12pt]{article}
\usepackage{epsfig,amssymb,amsfonts,amsmath}
\usepackage{e-jc}

\newenvironment{proof}{\noindent{\bf Proof.\,}}{\hfill$\Box$}
\newtheorem{theorem}{Theorem}
\newtheorem{corollary}[theorem]{Corollary}
\newtheorem{lemma}{Lemma}

\newtheorem{definition}{Definition}

\newtheorem{remark}{Remark}
\input epsf

\newcommand{\codeg}{{\rm codeg}}

 \font\xviiroman=cmr17
\def\udot{\mathbin{\ooalign{$\cup$\crcr
   \hfil\raise 5pt\hbox{\xviiroman\hspace{-0pt}.}\hfil\crcr}}}
\def\bigudotx#1#2{\mathop{\smash{\ooalign{$#1\bigcup$\crcr
   \hfil\raise 7pt\hbox{#2}\hfil\crcr}}\vphantom{\bigcup}}}

\date{\dateline{Feb 10, 2007}{Jan 4, 2008}}
\newcommand{\arrows}{\longrightarrow}
\newcommand{\notarrows}{\not\!\!\longrightarrow}

\begin{document}
\title{Avoiding rainbow induced subgraphs in vertex-colorings}
\author{Maria Axenovich and Ryan Martin\thanks{Research supported in part by NSA grant H98230-05-1-0257} \\
\small Department of Mathematics, Iowa State University,
 Ames, IA 50011 \\[-0.8ex]
\small \texttt{axenovic@iastate.edu, rymartin@iastate.edu}}
%\begin{document}
\maketitle

\begin{abstract}
For a fixed graph $H$ on $k$ vertices, and a graph $G$ on at least
$k$ vertices, we write $G\arrows H$ if in any vertex-coloring of $G$
with $k$ colors, there is an induced subgraph isomorphic to $H$
whose vertices have distinct colors. In other words, if $G\arrows H$
then a totally multicolored induced copy of $H$ is unavoidable in
any vertex-coloring of $G$ with $k$ colors.  In this paper, we show
that, with a few notable exceptions, for any graph $H$ on $k$
vertices and for any graph $G$ which is not isomorphic to $H$,
$G\notarrows H$. We explicitly describe all exceptional cases. This
determines the induced vertex-anti-Ramsey number for all graphs and
shows that totally multicolored induced subgraphs are, in most
cases, easily avoidable.
\end{abstract}

\section{Introduction}
Let $G=(V,E)$ be a graph. Let $c:V(G) \rightarrow [k]$ be a
vertex-coloring of $G$. We say that $G$ is \textit{monochromatic under $c$} if all vertices have the same color and we say that $G$ is \textit{rainbow} or \textit{totally multicolored under $c$} if all vertices of $G$ have distinct colors. The existence of a graph forcing an induced monochromatic subgraph isomorphic to $H$  is well
known. The following bounds are due to Brown and R\"odl:

\begin{theorem}[Vertex-Induced Graph Ramsey Theorem~\cite{BR}]
\label{vertex-induced} For all graphs $H$,  and all positive
integers $t$ there exists a graph  $R_t(H)$ such that if the
vertices of $R_t(H)$ are colored with $t$ colors, then there is an induced subgraph of $R_t(H)$ isomorphic to $H$ which is
monochromatic. Let the order of $R_t(H)$ with smallest number of vertices be $r_{mono}(t,H)$. Then there are constants $C_1=C_1(t),C_2=C_2(t)$ such that  $C_1k^2\leq\max\{r_{mono}(t,H): |V (H)| =
k\}\leq  C_2k^2 \log_2 k.$
\end{theorem}

Theorem \ref{vertex-induced} is one of numerous vertex-Ramsey
results investigating the existence of induced monochromatic
subgraphs, including the studies of Folkman numbers such as in \cite{LR}, \cite{BD} and others. There are also ``canonical''-type theorems claiming the existence of monochromatic or rainbow substructures (see, for example, a general survey paper by Deuber \cite{D}). The paper of Eaton and R\"odl provides the following specific result for vertex-colorings of graphs.
\begin{theorem}[Vertex-Induced Canonical Graph Ramsey
Theorem~\cite{ER}]
For all graphs $H$,  there is a graph  $R_{can}(H)$ such that if $R_{can}(H)$ is vertex-colored then there is an induced subgraph of $R_{can}(H)$ isomorphic to $H$ which is either monochromatic or rainbow. Let the order of such a graph with the smallest number of vertices be $r_{can}(H)$. There are constants $c_1, c_2$ such that $c_1 k^3 \leq \max\{r_{can}(H): |V(H)|=k\} \leq c_2 k^4 \log k$. \label{thm2}
\end{theorem}

In this paper, we study the existence of totally multicolored
induced subgraphs isomorphic to a fixed graph $H$, in any coloring
of a graph $G$ using exactly $k=|V(H)|$ colors. We call a coloring
of vertices, with $k$ nonempty color classes, a
\textit{$k$-coloring}. Whereas the induced-vertex Ramsey theory
minimizes the order of a graph that forces a desired induced
monochromatic graph, it is clear that for the multicolored case a
similar  goal is trivially achieved by the graph $H$ itself.  What is
not clear is whether it is possible to construct an arbitrarily
large graph $G$ with the property that any $k$-coloring of $V(G)$
induces a rainbow $H$.

\begin{definition}
   Let $G$ and $H$ be two graphs.  We say ``$G$ arrows $H$'' and write \textit{$G\arrows H$} if for any coloring of the vertices of $G$ with exactly $|V(H)|$ colors, there is an induced rainbow subgraph isomorphic to $H$.  Let
   $$ f(H)=\max\{|V(G)|:G\arrows H\} , $$
   if such a $\max$ exists.  If not, we write $f(H)=\infty$.
\end{definition}

It follows from the definition that  if $f(H)= \infty$ then for any $n_0\in {\mathbb N}$ there is  $n>n_0$ and a graph $G$ on $n$ vertices such  that any $k$-coloring of vertices of $G$ produces a rainbow induced copy of $H$. The function $f$ was first investigated by the first author in \cite{A}.

\begin{theorem}[\cite{A}]\label{A}
Let $H$ be a graph on $k$ vertices.  If $H$ or its complement is (1)
a complete graph, (2) a star or (3) a disjoint union of two adjacent
edges and an isolated vertex, then $f(H)=\infty$; otherwise
$f(H)\leq 4k-2$.
\end{theorem}

We improve the bound on $f(H)$ to the best possible bound on graphs $H$ for which $f(H)<\infty$.
\begin{theorem}
Let $H$ be a graph on $k$ vertices.  If $H$ or its complement is (1)
a complete graph, (2) a star or (3) a disjoint union of two adjacent
edges and an isolated vertex, then $f(H)=\infty$; otherwise
$f(H)\leq k+2$ if $k$ is even and $f(H)\leq k+1$ if $k$ is
odd.\label{summary}
\end{theorem}

What we prove in this paper is stronger.  First, we find $f(H)$ for
all graphs $H$.  Second, we are able to explicitly classify almost
all pairs $(G,H)$ for which $G\arrows H$. We describe some
classes of graphs and state our main result in the following
section.\\

\section {Main Result}

Let $K_n, E_n, S_n, C_n, P_n$ be a complete graph, an
empty graph, a star, a cycle and a path on $n$ vertices,
respectively.
Let $H_1+H_2$ denote the vertex-disjoint union of graphs $H_1$ and $H_2$. We denote  $\Lambda= P_3+K_1$.
  If $H$ is a graph, let $\overline{H}$ denote its
complement.  Let $P$ and $\Theta$ be the Petersen and
Hoffman-Singleton graphs, respectively; see Wolfram Mathworld
(\cite{MathWorldP} and \cite{MathWorldHS}, respectively) for
beautiful pictures.

Let $kH$ denote the vertex-disjoint union of $k$
copies of graph $H$.  We write $H\approx H'$ if $H$ is isomorphic to
$H'$ and we say that $H\in\{H_1,H_2,\ldots\}$ if there exists an
integer $i$ for which $H\approx H_i$.  We write $G-v$ to denote the
subgraph of $G$ induced by the vertex set $V(G)\setminus\{v\}$. A
graph is \textit{vertex-transitive} if, for every distinct
$v_1,v_2\in V(G)$, there is an automorphism, $\varphi$, of $G$ such
that $\varphi(v_1)=v_2$.  A
graph is \textit{edge-transitive} if, for every distinct
$\{x_1,y_1\},\{x_2,y_2\}\in E(G)$, there is an automorphism, $\varphi$, of $G$ such that either both $\varphi(x_1)=x_2$ and $\varphi(y_1)=y_2$ or both $\varphi(x_1)=y_2$ and $\varphi(y_1)=x_2$.

Let $P'$ and $\Theta'$ be the graphs obtained by
deleting two nonadjacent vertices from $P$ and $\Theta$,
respectively.  In the proof of Lemma \ref{regular}, we establish that both $\overline{P}$ and $\overline{\Theta}$ are edge-transitive, thus $P'$ and $\Theta'$ are well-defined.
 For $\ell\geq 3$, let $M_{\ell}$ denote a matching with $\ell$
edges;  let $M_{\ell}'$ denote the graph obtained by deleting two
nonadjacent vertices from $M_{\ell}$.  We say that a graph is
\textit{trivial} if it is either complete or empty. \\

We define several classes of graphs in order to prove the main
theorem.

Let $\mathcal{C}$ denote the class of connected graphs on
at least three vertices.

Let $\mathcal{P}'_3$ denote the set of graphs  $G=(V,E)$ such that
there is a nontrivial vertex-partition $V=V_1\cup V_2\cup V_3$, with
(a) $V_i\neq \emptyset$, for all $i=1,2,3$, (b) the tripartite
subgraph of $G$ obtained by deleting all edges with both endpoints in  $V_i$, $i=1,2,3$ is a vertex disjoint union of complete tripartite graphs and
bipartite graphs, each with vertices in only two of the parts $V_1,
V_2, V_3$;   see Figure \ref{P3}. Let $\mathcal{P}_3$ be the set of
all graphs on at least $4$ vertices which are not in
$\mathcal{P}'_3$.

\begin{figure}[h]
\hfill \begin{minipage}[t]{.45\textwidth}
\begin{center}
\epsfig{file=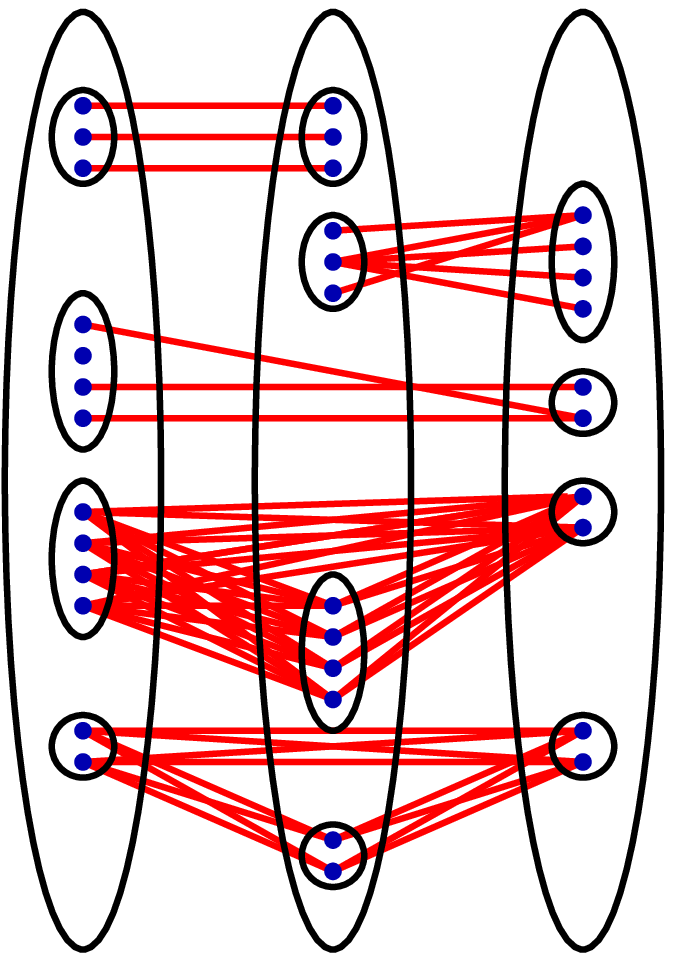, scale=0.5} \caption{A graph from class
$\mathcal{P}'_3$.} \label{P3}
\end{center}
\end{minipage}
\hfill
\begin{minipage}[t]{.45\textwidth}
\begin{center}
\epsfig{file=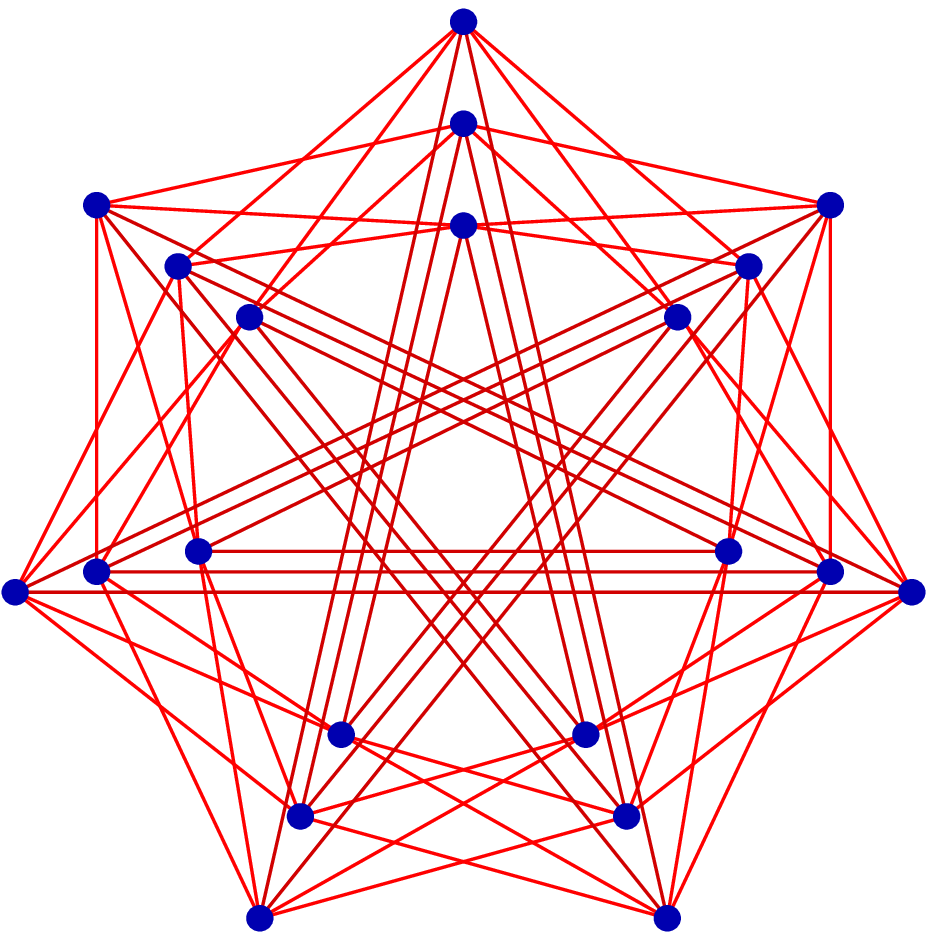, scale=0.55} \caption{Graph $G(3)$ which
arrows $\Lambda$.} \label{Lambda}
\end{center}
\end{minipage}
\hfill
\end{figure}

Let $\mathcal{L}=\{G(m): m\geq 1\}$, where $G(m)=(V,E)$, $V=\{v(i,j)
: 0\leq i\leq 6, 1\leq j\leq m\}$, $E=\{v(i,j)v(i+1, k) : 1\leq
j,k\leq m, j\neq k, 0\leq i\leq 6\}\cup\{v(i,j)v(i+3,j) : 1\leq
j\leq m, 0\leq i\leq 6\}$, addition is taken modulo $7$, see Figure
\ref{Lambda} for an illustration.

Let $\mathcal{T}$ denote the set of graphs $T$ such that (a) neither
$T$ nor $\overline{T}$ is complete or a star, and (b) either $T$ is
vertex-transitive or there exists a vertex, $v$ of degree $0$ or
$|V(T)|-1$ such that $T-v$ is vertex-transitive.  Note that a perfect matching is an example of a graph in $\mathcal{T}$.  If
$T\in\mathcal{T}$, denote $T'$ to be the graph that is obtained from
$T$ by deleting a vertex $w$ that is neither of degree $0$ nor
of degree $|V(T)|-1$. Let $\mathcal{T}'=\{T' : T\in\mathcal{T}\}$.  Note that,
given $T'\in\mathcal{T}'$, the corresponding graph $T\in\mathcal{T}$
is unique.

Let $\mathcal{F}_{\infty}=\left\{K_k,\overline{K_k} : k\geq
2\right\}\cup\left\{S_k,\overline{S_k} : k\geq
3\right\}\cup\{\Lambda,\overline{\Lambda}\}$.  As we see in Theorem
\ref{A}, $H\in\mathcal{F}_{\infty}$ iff $f(H)=\infty$. Observe (see also \cite{A}) that $G\arrows H$ if and only if
$\overline{G}\arrows\overline{H}$. In order to classify all graphs $G$ which arrow $H$, we introduce
the following notation $${\mathcal Arrow }(H)=\{G:  G\arrows H,
G\not\approx H\}.$$

\renewcommand{\labelenumi}{(\Alph{enumi})}
\renewcommand{\theenumi}{(\Alph{enumi})}
\begin{theorem}[Main Theorem]\label{main}

~\\
\begin{itemize}
\item
${\mathcal Arrow}(\Lambda) \supseteq \mathcal{L}$,\\
$~~{\mathcal Arrow}(K_k)=
\begin{cases}
{\mathcal C}&~ \mbox{ if }~k=2,\\
\{K_n: n> k\}&~ \mbox{ if }~k \geq 3,
\end{cases} $\\
$~~ {\mathcal Arrow}(S_k)=
\begin{cases}
\mathcal{P}_3& ~\mbox{ if }~k=3,\\
\{S_n: n>k\} &~\mbox{ if }~ k\geq 4,
\end{cases}$\\
\item
${\mathcal Arrow}(P')=\{P\}$,  $~~~{\mathcal Arrow}(\Theta')=\{\Theta\}$, \\
${\mathcal Arrow}(M'_\ell)= \{M_{\ell}, M_{\ell-1}+K_1\}$, $\ell\geq 3$,\\
\item
${\mathcal Arrow}(T')=\{T\}$, if $T'\in \mathcal{T'}$ and $T'\not\approx M_{\ell}'$, $\ell\geq 3$,\\
\item
If $H,\overline{H}\not\in\mathcal{F}_{\infty}\cup\{P', \Theta'\}\cup\mathcal{T}'$, then
${\mathcal Arrow}(H)=\emptyset$.\\
\end{itemize}
\end{theorem}

%For all graphs $H$, either $H$ or $\overline{H}$ is itemized in the following table:\\
%\begin{center}
%\begin{tabular}{||l|l|c||}\\ \hline
%$H$ & $\{G : G\arrows H, G\not\approx H\}$ & $f(H)$ \\ \hline\hline
%$K_2$ & $\mathcal{C}$ & $\infty$ \\ \hline
%$P_3$ & $\mathcal{P}_3$ & $\infty$ \\ \hline
%$K_k$, $k\geq 3$ & $\{K_n : n>k\}$ & $\infty$ \\ \hline
%$S_k$, $k\geq 4$ & $\{S_n : n>k\}$ & $\infty$ \\ \hline
%$\Lambda$ & $\mathcal{L}$ & $\infty$ \\ \hline
%$M_{k/2+1}'$ & $\left\{M_{k/2+1},M_{k/2}\cup K_1\right\}$ & $k+2$ \\ \hline
%$P'$ & $P$ & $10$($=k+2$) \\ \hline
%$\Theta'$ & $\Theta$ & $50$($=k+2$) \\ \hline
%$T'\in\mathcal{T}'$, $|V(T)|=k$ & $T\in\mathcal{T}$ & $k+1$ \\ \hline
%\end{tabular}
%\end{center}

\begin{corollary}
Let $H$ be a graph on $k$ vertices. Then

$$ f(H)=\begin{cases}
\infty, & H\in\mathcal{F}_{\infty},\\
k+2, & H\in\{P', \Theta', \overline {P'}, \overline{\Theta'} \}\cup \{M'_\ell, \overline {M'_\ell}: \ell \geq 3, k=2\ell-2 \}, \\
k+1, &  H\in \{T': T'\in \mathcal{T}'\}\setminus\{M'_\ell, \overline {M'_\ell}: \ell \geq 3, k=2\ell-2 \}, \\
k, & otherwise.
\end{cases} $$
\end{corollary}

\begin{remark} We wish to observe that a
graph $H$ for which $f(H)=k+2$ only occurs for even values of $k$,
$k\geq 4$ and is, up to complementation, uniquely defined by $k$
except in the cases of $k=8$ and $k=48$.    If $k\in\{8,48\}$, then
there are two such complementary pairs of graphs $H$.  We also note that ${\mathcal Arrow}(H)$ is fully classified for every graph $H$ except for $H\in\left\{\Lambda,\overline{\Lambda}\right\}$.
\end{remark}

\newcommand{\F}{\mathcal{F}}
%\begin{remark}
%   In Theorem \ref{main}, we have computed $f(H)$ for all graphs $H$.  Moreover, we have characterized the set $\left\{(G,H) : G\arrows H\right\}$ except when
%   $H\in\left\{\Lambda,\overline{\Lambda}\right\}$.
%Note that if $f(H)<\infty$, then the only graphs $H$ for which $f(H)>k+1$ are $H\in\left\{M_{k/2+1}',\overline{M_{k/2+1}'} : k\geq 4\right\}$ and $H\in\left\{P',\overline{P'},\Theta,\overline{\Theta}\right\}$.
%   Therefore, for $k\geq 4$ and $k\in\{8,48\}$ there are only two graphs $H$ (a complementary pair) on $k$ vertices such that $f(H)=k+2$.  But for $k\in\{8,48\}$, there are 4 such graphs (two complementary pairs).
% It is known, from \cite{A}, that for every $n$ which is a
%   positive integer multiple of $7$, there exists a $G_n$, of order $n$, such that $G_n\arrows\Lambda$.  We leave it to the reader to verify that, although $C_4\cup K_1\arrows\Lambda$, there is no graph $G$ on $6$ vertices such that $G\arrows\Lambda$.  Hence, regardless of the ultimate characterization of $\{G : G\arrows\Lambda\}$, the graphs $\Lambda,\overline{\Lambda}$ would not be appropriately classified in either \ref{caseA} or \ref{caseC}.
%Finally, we call the reader's attention to the fact that
%   $G\arrows H$ iff $\overline{G}\arrows\overline{H}$.  Hence,   characterizing the set $\left\{G : G\arrows\overline{H}\right\}$ is sufficient to determine the set $\{G : G\arrows H\}$.
%\end{remark}
\renewcommand{\labelenumi}{(\arabic{enumi})}
\renewcommand{\theenumi}{(\arabic{enumi})}

This paper is structured as follows: In Section \ref{Defs} we state without proofs all of the lemmas and supplementary results. In Section \ref{Proof}, we prove the main theorem. In Section \ref{LemmaProofs} we prove all the lemmas from Section \ref{Defs}.\\

The main technical tool of the proof is the fact that in most cases
we can assume that the degree sequence of the graph $H$ is
consecutive. Using this, it is possible to show that $f(H)\leq
|V(H)|+c$ for some absolute constant $c$ and for all $H$ such that
$f(H)<\infty$. We prove several additional cited lemmas which
provide a delicate analysis allowing one to get an exact result for
ALL graphs, in particular for ones with small maximum degree.\\

\section{{Definitions, Lemmas and supplementary results}}
\label{Defs}

Let $G$ be a graph on $n$ vertices and $v\in V(G)$.  The degree of $v$ is denoted $\deg(v)$ and the codegree of $v$, $n-1-\deg(v)$,  is denoted $\codeg(v)$. When the choice of a graph is ambiguous, we shall denote the degree of a vertex $v$ in graph $G$ by $\deg(G,v)$.  If vertices $u$ and $v$ are adjacent, we write $u\sim v$, otherwise we write $u\not\sim v$. For subsets of vertices $X$ and $Y$, we
write $X\sim Y$ if $x\sim y$ for all $x\in X$, $y\in Y$; we write $X\not\sim Y$ if $x\not\sim y$ for all $x\in X$, $y\in Y$. For a vertex $x\not\in Y$, we write $x\sim Y$ if $\{x\}\sim Y$ and $x\not\sim Y$ if $\{x\}\not\sim Y$.
For a subset $S$ of vertices of a graph $G$, let $G[S]$ be the subgraph induced by $S$ in $G$.
The neighborhood of a vertex $v$ is denoted $N(v)$, and the closed
neighborhood of $v$,  $N[v]=N(v)\cup \{v\}$.  We shall write $e(G)$
to denote the number of edges in a graph $G$. The subset of vertices
of degree $i$ in a graph $G$ is $G_i$. The minimum and maximum
degrees of a graph $G$ are denoted by $\delta(G)$ and  $\Delta(G)$,
respectively. For all other standard definitions and notations, see
\cite{W}.

We say the degree sequence of a graph $H$ is \textit{consecutive}
if, for every $i\in\{\delta(H),\ldots,\Delta(H)\}$, there exists a
$v\in V(H)$ such that $\deg(v)=i$. The following definition is
important and used throughout the paper.

\begin{definition}
For a graph $H$ on $k$ vertices, let the \textit{{\bf deck}} of $H$,
denoted ${\rm deck}(H)$, be the set of all induced subgraphs of $H$
on $k-1$ vertices. We say that a graph $F$ is in the deck of $H$ if
it is isomorphic to a graph from the deck of $H$. The graph $G$ on
$n$ vertices is said to be \textit{{\bf bounded by}} a graph $H$ on
$k$ vertices  if both $\Delta(G)=\Delta(H)$ and
$\delta(G)=n-k+\delta(H)$.
\end{definition}

For $S\subseteq V(G)$, if  $G[S]\approx H$, we say (to avoid lengthy notation), that
$S$ induces $H$ in $G$ and we shall label the vertices in $S$ as the
corresponding vertices of $H$.

We use the following characterization of regular graphs of diameter
$2$.
\begin{theorem}[Hoffman-Singleton, \cite{HS}]\label{HoffSingThm}
   If $G$ is a diameter 2, girth 5 graph which is $\Delta$-regular, then  $\Delta\in\{2,3,7,57\}$.  Moreover, if $\Delta=2$, $G$ is the $5$-cycle; if $\Delta=3$, then $G$ is the Petersen graph; and if $\Delta=7$, $G$ is the Hoffman-Singleton graph.  It is not known if such a graph exists for $\Delta=57$.
\end{theorem}

Note that if a $57$-regular graph of diameter $2$ exists,  it is called a
$(57,2)$-Moore graph.

One of our tools is the following theorem of  Akiyama,  Exoo and
Harary \cite{AEH}, later strengthened by Bos\'ak \cite{B}.
\begin{theorem}[Bos\'ak's theorem] \label{trivial}
   Let $G$ be a graph on $n$ vertices such that all induced
   subgraphs of $G$ on $t$ vertices have the same size.  If
   $2\leq t\leq n-2$ then $G$ is either a complete graph or an empty
   graph.
\end{theorem}

\vskip 0.3cm
In all of the lemmas below we assume that
$$|V(G)|=n, \quad |V(H)|=k, \quad \Delta=\Delta(H), \mbox{ and }  \delta=\delta(H).$$

%{\footnote {bounded}}
\begin{lemma}\label{bounded}
   If $G\arrows H$, then the following holds:\\
   \begin{enumerate}
      \vspace{-11pt}\item If $\Delta \leq k-3$, then $\Delta(G)=\Delta$. \label{lem:deg}\\
      \vspace{-11pt}\item If $2\leq\delta\leq\Delta\leq k-3$, then $n\leq k+\Delta-\delta$ with equality iff $\Delta(G)=\delta(G)$. \label{cor:deg}
   \end{enumerate}
\end{lemma}

%{\footnote {consecutive-degrees}}
\begin{lemma}\label{consecutive-degrees}
  If $H$ is a graph on $k\geq 3$ vertices and $G$ is a graph on $n\geq k+2$ vertices such that $G\arrows H$, then either $H$ or its complement is a star or the degree sequence of $H$ is consecutive.
\end{lemma}

The Deck Lemma is an important auxiliary lemma that is used
throughout this paper.
%\footnote{deck}
\begin{lemma}[Deck lemma]\label{deck}
   Let $G\arrows H$.  For any set $U\subset V(G)$ with $|U|=k-1$, $G[U]$ is in the deck of $H$.  Consequently, $e(H)-\Delta\leq e(G[U])\leq e(H)-\delta$.
\end{lemma}

%\footnote{Delta-delta}
\begin{lemma}\label{Delta-delta}
   If $f(H)>k$ and $H$ has consecutive degrees, then $\Delta\leq\delta+3$.
\end{lemma}

Observe that Lemmas \ref{bounded}, \ref{consecutive-degrees} and
\ref{Delta-delta} immediately imply that $f(H)\leq |V(H)|+3$ if
$2\leq \delta\leq \Delta\leq k-3$. The remaining lemmas allow us to
deal with the cases where $\delta<2$ or $\Delta>k-3$ and to prove
exact results.

Lemmas \ref{infinity} and \ref{k+1}  address the cases where $f(H)=\infty$ and
$f(H)=k+1$.

\begin{lemma}\label{infinity}
$~~{\mathcal Arrow}(K_k)=
\begin{cases}
{\mathcal C},&k=2,\\
\{K_n: n> k\},&k \geq 3;
\end{cases} $ \hskip 20pt
$~~ {\mathcal Arrow}(S_k)=
\begin{cases}
\mathcal{P}_3,& k=3,\\
\{S_n: n> k\}, &k\geq 4.
\end{cases}$
\end{lemma}

\begin{lemma}\label{k+1}
$\{(G,H): G\arrows H, |V(G)|=k+1\} =\{(T,T') : T\in
\mathcal{T}\}$.
\end{lemma}

Lemmas \ref{regular} and \ref{tool} allow us to deal with the case where $n\geq
k+2$ and $G$ is regular or almost regular.
%\footnote{regular}
\begin{lemma}\label{regular}
Assume that  $k\geq 3$. Let $Q$ be the set of pairs $(G,H)$ such that  $|V(G)|\geq k+2$, $G\arrows H$, $G$ is
bounded by $H$,  $H\not\in \mathcal{F}_{\infty}$,  $H$ has
consecutive degrees and $G$ is $d$-regular for some $d\geq 2$. Then  $Q=\left\{(P,P'), (\overline{P},\overline{P'}), (\Theta,\Theta'), (\overline{\Theta},\overline{\Theta'})\right\}$.
\end{lemma}

%\footnote{tool}
\begin{lemma}\label{tool}Let $|V(G)|=k+2$ and let $G$ be bounded by $H$.  If $\Delta-\delta=3$
and
   $\Delta(G)-\delta(G)=1$, then $G\notarrows H$.
\end{lemma}

The following is a technical lemma used in the proof of the Main Theorem and Lemma \ref{03}.

%\footnote{unbounded-unique-delta}
\begin{lemma}\label{unbounded-unique-delta}
If $|V(G)|\geq k+2$, $G\arrows H$, $\delta=1$,
and $\delta(G)<n-k+\delta$, then $\Delta\leq\delta+2=3$.  Furthermore,
if equality holds, then $|H_3|=1$, $H_3\sim H_2$, and there is an $S\subseteq V(G)$ and $v\in V(G)\setminus S$
such that  $G[S]\approx H$,  $|N(v)\cap S|=1$ and $v\not\sim H_3\cup H_2$.
\end{lemma}

Finally, the following lemmas treat the case when
$\Delta=\Delta(H)\in \{1,2,3\}$.

\begin{lemma} \label{matching-isolated}
Let $\Delta=1$, $H\not\in \mathcal{F}_{\infty}$ and  $|V(G)|\geq
k+2$. Then $G\arrows H$ implies that $k$ is even and $(G,H)=
(M_{k/2+1}, M'_{k/2+1})$.
\end{lemma}

\begin{lemma}\label{Delta2}
Let $\Delta=2$, $H\not\in\mathcal{F}_{\infty}$, $|V(G)|\geq k+2$
and $\delta(G)<n-k+\delta$. Then, $G\notarrows H$.
\end{lemma}

\begin{lemma}\label{03}
Let $\Delta=3$, $H\not\in \mathcal{F}_{\infty}$, $|V(G)|\geq k+2$
and $\delta(G)<n-k+\delta$. Then, $G\notarrows H$.
\end{lemma}~\\

%\begin{remark}
%\renewcommand{\labelenumi}{(\Alph{enumi})}
%\renewcommand{\theenumi}{(\Alph{enumi})}
%   Each of the graphs $H$ in cases \ref{caseA}, \ref{caseB},
%   \ref{caseC}, \ref{caseE}, and \ref{caseF} is a graph $H$ as
%   described by Lemma \ref{k+1}.  Interestingly, only
%   $H\approx P',\overline{P'},\Theta',\overline{\Theta'}$ are graphs $H$
%   on $k$ vertices for which $f(H)>k$ but there is no graph $G$ on
%   $k+1$ vertices such that $G\arrows H$.
%\renewcommand{\labelenumi}{(\arabic{enumi})}
%\renewcommand{\theenumi}{(\arabic{enumi})}
%\end{remark}~\

\section{{\bf PROOF of the MAIN THEOREM}}\label{Proof}

Let $H$ be a graph on $k$ vertices. Recall that
$\mathcal{F}_{\infty}=\left\{K_k,\overline{K_k} : k\geq
2\right\}\cup\left\{S_k,\overline{S_k} : k\geq
3\right\}\cup\left\{\Lambda,\overline{\Lambda}\right\}$. If $H\in
\mathcal{F}_{\infty}$, then the theorem follows from Lemma \ref{infinity} and Theorem~\ref{A}.

Let $G\arrows H$, $|V(G)|>k$ and $H\not\in\mathcal{F}_{\infty}$. We
shall describe all such graphs $G$ on $n$ vertices.

If   $n=k+1$, then
Lemma \ref{k+1} claims that $H\approx T'\in \mathcal {T}'$ and $G\approx T$.
Note that $M_{\ell}'\in\mathcal{T}'$ for all $\ell\geq 3$.  If $T'=M_{\ell}'$, then $T=M_{\ell-1}+K_1$.  Therefore we may assume that $n\geq k+2$
and  $H\not\in\mathcal{F}_{\infty}$. By Lemma
\ref{consecutive-degrees}, the degree sequence of $H$ is
consecutive.\\

\noindent\textbf{CASE 1.} $G$ is bounded by $H$.\\

\indent Recall that $G$ being bounded by $H$ means that
$\Delta(G)=\Delta$ and $\delta(G)=n-k+\delta$.  By Lemma
\ref{Delta-delta}, $\Delta\leq\delta+3$.   Lemma \ref{bounded} gives
that $n\leq k+3$.

First, suppose $G$ is $\Delta$-regular.  If $\Delta\geq 2$, then by
Lemma \ref{regular},
$G\in\left\{P,\overline{P},\Theta,\overline{\Theta}\right\}$ and
$n=k+2$.  If $\Delta\leq 1$, then $G$ is a matching.  Lemma
\ref{matching-isolated} covers this case and gives that $H\approx
M_{k/2+1}'$.

Second, suppose $G$ is not regular, then
$$ n-k+\delta=\delta(G)<\Delta(G)=\Delta . $$
Since $\Delta-\delta\leq 3$,  Lemma \ref{bounded} implies that
$n-k<3$.  The fact that  $n\geq k+2$, implies that $n=k+2$. Applying
Lemma \ref{bounded} again, we see that $\Delta-\delta=3$ and
$\delta(G)= (n-k)+\delta = 2+ (\Delta-3) = \Delta -1$. Thus
$\Delta(G)-\delta(G)=1$.  By Lemma \ref{tool}, $G\notarrows H$, a
contradiction.\\

\noindent\textbf{CASE 2.} $G$ is not bounded by $H$.\\

By Lemma \ref{bounded}, if $G\arrows H$ and $G$ is not bounded by
$H$, then either $\delta(H)\leq 1$ (in the case where
$\delta(G)<n-k+\delta$) or $\Delta(H)\geq k-2$ (in the case where
$\Delta(G)>\Delta$). Using the fact that $G\arrows H$ iff
$\overline{G}\arrows\overline{H}$, we will assume, without loss of
generality, that $\delta(G)<n-k+\delta$ and $\delta\leq 1$.

Using Lemma \ref{unbounded-unique-delta} (when $\delta=1$) and Lemma \ref{Delta-delta} (when $\delta=0$),   we have that
$\Delta\leq 3$.  Since $\Delta\in\{1,2,3\}$,  Lemmas \ref{matching-isolated}, \ref{Delta2}, \ref{03}
give that $(G,H)=(M_\ell,
M_\ell')$.\\

Summarizing CASES 1 and 2, we see that if $n\geq k+2$ and
$H\not\in\mathcal{F}_{\infty}$, then $n=k+2$ and
$H$ or $\overline{H}$ is in $\{M_{k/2+1}',P', \Theta'\}$.  Lemma \ref{matching-isolated} and the fact that $M_{\ell}'\in\mathcal{T}'$ for all $\ell\geq 3$ give that ${\mathcal Arrow}\left(M_{\ell}'\right)=\left\{M_{\ell},M_{\ell-1}+K_1\right\}$.  Lemma \ref{regular} gives that ${\mathcal Arrow}(P')= \{P\}$ and ${\mathcal Arrow}(\Theta')=\{\Theta\}$.

This concludes the proof of Theorem \ref{main}.\\

\section{Proofs of Lemmas}
\label{LemmaProofs}
\renewcommand{\labelenumi}{(\arabic{enumi})}
\renewcommand{\theenumi}{(\arabic{enumi})}

\subsection{Proof of Lemma \ref{bounded}}~\\
\indent\ref{lem:deg}  Since $G\arrows H$, $\Delta(G)\geq\Delta$.
Let $\Delta\leq k-3$.  Suppose there exists a vertex $v\in V(G)$ such that $\deg(v)>\Delta$.  Color $N(v)$ with the first $\Delta+1+a$ colors, where $a$ is the largest integer such that both $\Delta+1+a\leq\deg(v)$ and $\Delta+1+a\leq k-1$. Color $v$ with color $k$ and color the rest of the vertices (if such exist) with the remaining colors (or color these vertices with color $1$ if no colors remain). Any $S\subseteq V(G)$ that induces a rainbow copy of $H$ has a vertex, namely $v$, of degree greater than $\Delta$, a contradiction.\\

\indent\ref{cor:deg}  By Part \ref{lem:deg}, $\Delta(G)=\Delta$. We have that $\Delta(\overline{H})=k-1-\delta(H)\leq k-3$. Hence,
$n-1-\delta(G)=\Delta(\overline{G})=\Delta(\overline{H})=k-1-\delta$.
So, $\delta(G)=n-k+\delta$ and
$\Delta=\Delta(G)\geq\delta(G)=n-k+\delta$.  Thus, $n\leq
k+\Delta-\delta$ with equality if and only if
$\Delta(G)=\delta(G)$.
$\Box$ \\~\\

\subsection{Proof of Lemma \ref{consecutive-degrees}}~\\
Let $H$ have the property that there is an $i$,
$\delta(H)<i<\Delta(H)$ such that there is  no $w\in V(H)$ with  $\deg(w)=i$.  Let $L_i(H)=\{v\in V(H) : \deg(v)<i\}$, and $U_i(H)=\{v\in V(H) : \deg(v)>i\}$.  Let $L_i(G)=\{v\in V(G) : \deg(v)<i\}$, and let $U_i(G)=\{v\in V(G) : \deg(v)>n-k+i\}$. Since $G\arrows H$, we may assume that $H\subseteq G$.\\

\noindent\textbf{Claim 1.} $V(H)=L_i(H)\cup U_i(H)$ and
$V(G)=L_i(G)\cup U_i(G)$.\\
\indent The first statement of the claim follows from our assumption on $H$. Assume that there is a vertex $v\in V(G)$ with $i\leq\deg(v)\leq n-k+i$.  Color $v$ with one color, $N(v)$ with $i$ other colors and $V(G)\setminus N[v]$ with the remaining $k-i-1$ colors. Any induced rainbow subgraph $H'$ of $G$ on $k$ vertices must contain $v$ and exactly $i$ of its neighbors.  Thus $H'$ can not be isomorphic to $H$; i.e., $G\notarrows H$, a contradiction.  This proves Claim 1.\\

\noindent\textbf{Claim 2.} $U_i(H)\subseteq U_i(G)$ and
$L_i(H)\subseteq L_i(G)$.\\
\indent If there is a vertex $w\in U_i(H)\cap L_i(G)$, then
$\deg(G,w)\leq i-1<i+1\leq\deg(H,w)$, a contradiction.  If there is a vertex $w\in L_i(H)\cap U_i(G)$, then $\deg(H,w)\leq i-1$, $\deg(G,w) \geq n-k+i+1$. Thus, $\codeg(H,w) \geq k-i$ and $\codeg(G,w)\leq k-i-2$, a contradiction since $\codeg(G,u)\geq\codeg(H,u)$ for all $u\in V(H)$.  This proves Claim 2.\\

Assume first that $|U_i(H)|=|U_i(G)|=1$ and consider an arbitrary $(k-1)$-subset
$U\subseteq L_i(G)$.  Color the vertices of $U$ with $k-1$ colors
and color the rest of $V(G)$ with the remaining color.  The induced
copy of $H$ must contain the member of $U_i(G)$ and so $U\cup
U_i(G)$ must induce $H$.  We may conclude that all $(k-1)$-subsets
of $L_i(G)$ are isomorphic. Since $|L_i(G)|=n-1\geq k+1$, Bos\'ak's
theorem implies that $L_i(G)$ induces a
trivial subgraph.  Given that $U\cup U_i(G)$ must induce $H$ for any
such $U$ and the degree sequence is not consecutive, both $G$ and
$H$ must be stars.

Now assume that  $|U_i(G)|\geq 2$ and $|U_i(H)|=1$.  Color as many vertices of
$U_i(G)$ with distinct colors as possible (at least two, at most
$k-1$) and color the rest with the remaining colors.  Under this
coloring, any rainbow subgraph on $k$ vertices will have at least
$2$ vertices in $U_i(G)$, a contradiction to Claim 2.

Thus, we may assume that $|U_i(H)|\geq 2$ and a complementary
argument implies that $|L_i(H)|\geq 2$.  Since $n\geq k+2$, it is
the case that either $|U_i(G) |>|U_i(H)|$ or $|L_i(G)|>|L_i(H)|$.
Without loss of generality, assume the former.  We know that
$|U_i(H)|=k-|L_i(H)|\leq k-2$.  Color $U_i(G)$ with $|U_i (H)|+1\leq
k-1$ colors and $L_i(G)$ with the remaining colors.  Under this
coloring, any rainbow subgraph of $G$ will have more than $|U_i(H)|$
vertices in $U_i(G)$, a contradiction to Claim 2. $\Box$ \\~\\

\subsection{Proof of Lemma \ref{deck}}~\\
Consider a $(k-1)$-subset $U\subseteq V(G)$. Color its vertices with
$k-1$ distinct colors and color the rest of the vertices with the remaining color.  Since there is a rainbow copy of $H$ in this coloring,
and its vertices must contain $U$, $G[U]$ must be in the deck of
$H$. Since each $(k-1)$-vertex induced subgraph of $H$ has at least
$e(H)-\Delta$ and at most $e(H)-\delta$ edges, the second statement
of the lemma follows. $\Box$ \\~\\

\subsection{An important auxiliary lemma}~\\
Recall that $H_d=\{w\in V(H) : \deg(H,w)=d\}$.

%\footnote{neighborhood-into-S}
\begin{lemma} \label{neighborhood-into-S}
Let $H$ be a graph on $k$ vertices with consecutive degrees and let $G$ be
a graph on $n\geq k+1$ vertices such that $G\arrows H$. Furthermore,
let $S=\{y_1,y_2,\ldots,y_k\}\subseteq V(G)$ such that $G[S]\approx
H$. Let
$\deg(G[S],y_1)\leq\deg(G[S],y_2)\leq\cdots\leq\deg(G[S],y_k)$. Each
of the following is true:\\
\begin{enumerate}
   \vspace{-11pt}\item For any $v\in V(G)\setminus S$,
   $|N(v)\cap(S\setminus\{y_k,y_{k-1}\})|\geq\Delta-2$. If
   equality holds, then $|H_{\Delta}|=1$ and $H_{\Delta}\sim
   H_{\Delta-1}$.  If $H_{\Delta}\supseteq\{y_k,y_{k-1}\}$ and
   $y_k\not\sim y_{k-1}$ then for any $v\in V(G)\setminus S$,
   $|N(v)\cap(S\setminus\{y_k,y_{k-1}\})|\geq\Delta$.\label{DEL}\\
   \vspace{-11pt}\item For any $v\in V(G)\setminus S$,
   $|N(v)\cap(S\setminus\{y_1,y_2\})|\leq\delta+1$. If equality
   holds, then $|H_{\delta}|=1$ and $H_{\delta}\not\sim
   H_{\delta+1}$.  Moreover, if $H_{\delta}\supseteq\{y_1,y_2\}$ and
   $y_1\sim y_2$ then for any $v\in V(G)\setminus S$,
   $|N(v)\cap(S\setminus\{y_k, y_{k-1}\})|\leq \delta-1$.\label{del}\\
   \vspace{-11pt}\item There is a vertex $v\in V(G)\setminus S$ such that either $\{v\}\cup
   S\setminus\{y_k\}$ induces $H$ or $\{v\}\cup S\setminus\{y_1\}$
   induces $H$.\label{DELdel}
\end{enumerate}
\end{lemma}

\begin{proof}~\\
   \indent\ref{DEL} Let $U=\{v\}\cup S\setminus\{y_k,y_{k-1}\}$.
   Using the Deck Lemma  and counting edges incident to
   $y_k$ and $y_{k-1}$, we have $e(H)-\Delta\leq e(G[U])\leq
   e(H)-\Delta-(\Delta-1)+1+|N(v)\cap(S\setminus\{y_k, y_{k-1}\})|$.
   It follows that
   $$ \left|N(v)\cap\left(S\setminus\{y_k, y_{k-1}\}\right)\right|\geq
      \Delta-2 . $$
   If $y_k\not\sim y_{k-1}$ and both $y_k$ and $y_{k-1}$ are of
   degree $\Delta$, then $|N(v)\cap(S\setminus\{y_k,
   y_{k-1}\})|\geq\Delta$.\\

   \indent\ref{del} Let $U=\{v\}\cup S\setminus\{y_1,y_2\}$. Then
   $e(H)-\delta\geq e(U)\geq
   e(H)-\delta-(\delta+1)+|N(v)\cap(S\setminus\{y_1,y_2\})|$.
   Thus, all the statements in this part hold similarly to part
   \ref{DEL}.\\

   \indent\ref{DELdel} Rainbow color $S\setminus\{y_1,y_k\}$ with
   colors $\{1,\ldots,k-2\}$, both of the vertices in $\{y_1,y_k\}$ with
   color $k-1$ and $V(G)\setminus S$ with color $k$.  Regardless of
   which vertex of color $k-1$ is chosen, the statement holds.
\end{proof}~\\

\subsection{Proof of Lemma \ref{Delta-delta}}~\\
Let $G\arrows H$, $|V(G)|>k$, $S\subseteq V(G)$ and $G[S]\approx H$.
Lemma \ref{neighborhood-into-S} part \ref{DELdel} implies two cases:
\\

\noindent\textbf{CASE 1.} There is a $v\in V(G)\setminus S$ so that $S\cup\{v\}\setminus\{y_k\}$ induces $H$.\\
\indent Consequently, $|N(v)\cap S|\geq\Delta$ and, in particular,
$\Delta-2\leq\left|N(v)\cap\left(S\setminus\{y_1,y_2\}\right)\right|\leq\delta+1$.
The last inequality follows from Lemma \ref{neighborhood-into-S}
part \ref{del}.\\

\noindent\textbf{CASE 2.} There is a $v\in V(G)\setminus S$ so that $S\cup\{v\}\setminus\{y_1\}$ induces $H$.\\
\indent Consequently, $|N(v)\cap S|\leq\delta+1$ and
$\delta+1\geq\left|N(v)\cap\left(S\setminus\{y_k,y_{k-1}\}\right)\right|\geq\Delta-2$.
The last inequality follows from Lemma \ref{neighborhood-into-S}
part \ref{DEL}. \\

In both cases  $\Delta-\delta\leq 3$. $\Box$ \\~\\

\subsection{Proof of Lemma \ref{infinity}}~\\
If $H=K_2$ and  $G$ is disconnected, then color the vertices in one
component of $G$  with color $1$ and all other vertices with color
$2$. Thus $G\notarrows K_2$. On the other hand, if $G\notarrows
K_2$, then there is a partition of $V(G)=V_1\cup V_2$ such that
$V_1\not\sim V_2$.\\

If $H=K_k$, $k\geq 3$ and $G\neq K_n$, $n>k$, then $G\notarrows H$ follows
from the Deck Lemma since $G$ has two nonadjacent vertices or $n<k$. On the
other hand, it is obvious that $K_n\arrows K_k$, for all $n\geq k$.\\

Let $H= S_k$ for $k\geq 4$.
Then by the Deck Lemma, we see that $G$ has no induced subgraph isomorphic to $\overline{P_3}$ and no $K_3$.
Thus $\overline{G}$ has no induced $P_3$, and therefore $\overline{G}$ is a vertex disjoint union of cliques,
which implies that $G$ is a complete multipartite graph. Since $G$ has no $K_3$, $G$ is a complete bipartite graph.
If both parts of $G$ contain at least $2$ vertices, color the vertices in these parts with disjoint sets of colors
such that each part uses at least two colors. Then any rainbow $k$-subgraph is a complete bipartite graph with at least
two vertices in each part, a contradiction. So, we conclude that $G$ has only one vertex in one of the parts, thus $G$ is a star.\\

Let $H=P_3$.  It is easy to see that if $G \not\in \mathcal{P}_3$,
then the tri-partition $V_1, V_2, V_3$  of $V(G)$ as in the
definition of $\mathcal{P}'_3$   witnesses that $G\notarrows P_3$ by
coloring $V_1,V_2,V_3$ each with distinct colors. Suppose there is a
coloring of $V(G)$ with no rainbow copy of $P_3$.  Let the color
classes be $V_1,V_2,V_3$. Let $G'$ be a  tripartite subgraph of $G$
with parts $V_1,V_2,V_3$ which is obtained from $G$ by deleting all edges
with both endpoints in $V_i$, $i=1,2,3$. Consider a connected
component $Q$  of $G'$  with vertices in all three parts
$V_1,V_2,V_3$. We claim that this component is a complete tripartite
graph. To see this, consider the maximal complete tripartite
subgraph $Q'$ of $Q$. It is clear that $Q$ has a path with one
vertex in each of $V_1,V_2,V_3$.  This path must
induce a triangle; so $Q\neq\emptyset$.  If $Q'\neq Q$, then there
is a vertex $v\in V(Q)\setminus V(Q')$ such that $v$ is adjacent to
a vertex in $Q'$.  Without loss of generality assume that $v\in
V_1$, then $v$ must be adjacent to all vertices of $Q'$ in $V_2$ and
$V_3$. Thus $Q'\cup \{v\}$ is a complete tripartite graph larger
than $Q'$, a contradiction.  So, $Q'=Q$ and $Q$ is a complete
tripartite graph. Therefore, each component of $G'$ either has
vertices in only two parts or is a complete tripartite graph, so
$G\not\in \mathcal{P}_3$. $\Box$\\~\\

\subsection{Proof of Lemma \ref{k+1}}~\\
Let $G\arrows H$ and $n=k+1$.  Any coloring of $V(G)$ with $k$
colors assigns the same color to some two vertices.  Thus, for any
$u,v\in V(G)$, either $G-u$ or $G-v$ is isomorphic to $H$.  As an
immediate consequence, for at least $n-1$ vertices in $G$, the vertex
degrees have the same value $d=e(G)-e(H)$.  As a result, there are
only three possibilities: \\

\noindent\textbf{CASE 1.}   $G-w\approx H$ for all $w\in V(G)$.\\
\indent In particular, $G-u\approx G-v$ for all $u, v\in V(G)$.  Then, $G$ is regular. Since an isomorphism from $G-u$ and $G-v$ can be extended to an automorphism of $G$ mapping $u$ to $v$, we see that $G$ is vertex-transitive.\\

\noindent\textbf{CASE 2.}  There is exactly one vertex, $v$, such that $G-v\not\approx H$ and $\deg(G,v)=d'\notin\{0,n-1\}$.\\
\indent As before, we have that  for some $d$,   $\deg(G,w)=d$  for all $w\in
V(G)\setminus\{v\}$.  If $d'>d$, then the deletion of a $w\in
V(G)\setminus N[v]$ gives exactly one vertex of degree $d'$ and the rest of degree $d$ or $d-1$, but the deletion of a neighbor of $v$ does not, a contradiction. Similarly, if $d'<d$, then the deletion of a $w\in N(v)$ gives exactly one vertex of degree $d'-1$ but the deletion of a nonneighbor of $v$ does not, a contradiction.  Thus $d'=d$.

Let $w\sim v$, $w'\not\sim v$. Let $\varphi$ be an isomorphism from
$G-w$ to $G-w'$. Then, $\varphi$ maps vertices of degree $d-1$ in
$G-w$, which correspond to the neighbors of $w$, to vertices of
degree $d-1$ of $G-w'$, which correspond to the neighbors of $w'$.
In particular, $\varphi$ maps $v$ to some vertex $x\neq v$. As
before, we can extend $\varphi$ to an automorphism of $G$ by mapping
$w$ to $w'$. The existence of this automorphism
implies that $G-v\approx G-x$, and we can apply {CASE 1}.\\

\noindent\textbf{CASE 3.} There is exactly one vertex, $v\in V(G)$ such that $G-v\not\approx H$ and $\deg(G,v)\in\{0,n-1\}$.\\
\indent Assume without loss of generality that $\deg(G,v)=0$.  Then
$G-v-u\approx G-v-w$ for all $u,w \in V(G-v)$. As in CASE 1, $G-v$
is vertex transitive. \\

The above implies that if $G\arrows H$, $|V(G)|=k+1$ then $(G,H)=(T,T')$ for some $T\in \mathcal{T}$.
Now, let $T'\in \mathcal{T}'$. We need to show that $T\arrows T'$.  Let $|V(T')|=k$.  If we color the vertices of $T$ with $k$ colors then exactly two vertices, say $u$ and $v$ get the same color and the rest are totally multicolored.  So, if $T$ is vertex-transitive, then $T-\{u\}\approx T'$ and it is rainbow;
if $T$ is a union of a vertex transitive graph and an isolated vertex $w$, then without loss of generality $u\neq w$ and
$T-\{u\}\approx T'$ and it is rainbow; if $T$ has a vertex of degree $k$, the result follows from the previous case by considering $\overline{T}$. $\Box$\\~\\

\subsection{Proof of Lemma~\ref{regular}}~\\
Since $G$ is bounded by $H$, \begin{equation}
   n-k+\delta=\delta(G)=\Delta(G)=\Delta .
   \label{delbound}
\end{equation}
By Lemma~\ref{Delta-delta}, $\Delta-\delta\leq 3$.  Therefore,
either $n-k=2$ or $n-k=3$.\\

\noindent\textbf{CASE 1.} $n-k=3$. \\
\indent In this case, inequality (\ref{delbound}) becomes
$$ 3+\delta=\delta(G)= \Delta(G)=\Delta\leq \delta+3 . $$
Thus, all the inequalities are equalities, so $\delta=\Delta-3$.
This implies that $k\geq 4$. If $k=4$, then $\Delta=3$, $\delta=0$, and the fact that the degrees are consecutive implies that
$H\approx \overline{\Lambda}$, a contradiction to the assumption that $H\not\in \mathcal{F}_{\infty}$.
Thus, we can assume that $k\geq 5$.

Let $G[S]\approx H$. Let $y,y',w,w'$ be vertices in $S$ with degrees
$\Delta-3,\Delta-2,\Delta-1,\Delta$, respectively, in $G[S]$. The
fact that $G$ is $\Delta$-regular gives that $y,y',w,w'$ are
adjacent to $3,2,1,0$ vertices, respectively, in $V(G)\setminus S$.
Note that $y,y'$ and some two vertices in $V(G)\setminus S$ span
$C_4$ in $G$ and $w,w'$ and some two vertices in $V(G)\setminus S$
span $\overline{C_4}$ in $G$.

Let $U$ be a set of vertices in $G$ spanning a $C_4$. Color all
vertices of $U$ with color 1 and rainbow color the remaining
vertices with colors $\{2,\ldots,k\}$. Under this coloring, there is a rainbow copy of $H$ containing exactly one vertex of $U$.  Thus, the set of vertices outside of this copy of $H$ spans at least $2$ edges.  As a result, $e(H)\geq e(G)-3\Delta+2$.

Let $U'$ be a set of vertices spanning $\overline{C_4}$ in $G$.
Color all vertices of $U'$ with color 1 and rainbow color the
remaining vertices with colors $\{2,\ldots,k\}$. Under this
coloring, there is a rainbow copy of $H$ containing exactly one
vertex of $U'$.  Thus, the set of vertices outside of this copy of $H$ spans at most $1$ edge. As a result, $e(H)\leq e(G)-3\Delta+1$,
a contradiction to the
bound of $e(H)\geq e(G)-3\Delta+2$, derived above.\\

\noindent\textbf{CASE 2.} $n-k=2$. \\
In this case, inequality (\ref{delbound}) becomes
\begin{equation}\label{delbound2}
   2+\delta=\delta(G)= \Delta(G)=\Delta .
\end{equation}
Thus $\delta=\Delta-2$. Let $G[S]\approx H$.

If $n\leq 5$, then $k\leq 3$ so $H\in \mathcal{F}_{\infty}$. Thus,
we may assume that $n\geq 6$. By Ramsey's theorem, either $G$
contains a $K_3$ or $G$ contains $\overline{K_3}$. Without loss of
generality,
assume that $G$ contains $\overline{K_3}$.\\

\noindent\textbf{Claim 1.} For any copy of $H$ in $G$, the two vertices outside of it are not adjacent and $e(H)=e(G)-2\Delta$.\\
\indent By coloring the vertices of some copy of $\overline{K_3}$
with the color $1$ and rainbow coloring the rest of $V(G)$ with
colors $\{2,\ldots,k\}$, it is clear that $e(H)=e(G)-2\Delta$.  If
there are two adjacent vertices outside of a copy of $H$ in $G$ then
$e(H)=e(G)-2\Delta+1$, a contradiction.\\

\noindent\textbf{Claim 2.} $G$ has diameter $2$. \\
\indent If $G$ has diameter greater than $2$, then either there are
two vertices at a distance $3$ in $G$ or $G$ is disconnected.  If
$G$ is disconnected, then color two vertices in one component with
color $1$ and two vertices in the other component with color $2$ and
rainbow color the remaining vertices with colors $\{3,\ldots,k\}$.
No matter which vertices are chosen, any rainbow graph on $k$
vertices, under this coloring, has all vertices of degree $\Delta$
or $\Delta-1$.

If $G$ is connected and of diameter at least $3$, then let $u$ and
$v$ be vertices at distance exactly $3$ and $(u,x,y,v)$ be a
shortest $u$-$v$-path. Color $u$ and $y$ with color $1$, color $v$
and $x$ with color $2$, and rainbow color the remaining vertices
with colors $\{3,\ldots,k\}$. Under this coloring, a rainbow copy of
$H$ must contain $x$ and $y$; otherwise, there is an edge outside of
a copy of $H$, contradicting Claim 1. Therefore, $u$ and $v$ are outside of $H$. But $u$
and $v$ do not have a common neighbor, so $H$ has only vertices of
degree $\Delta$ and $\Delta-1$.

This contradicts the fact that $\delta=\Delta-2$. \\

\noindent\textbf{Claim 3.} $G$ has no $K_3$ and no $C_4$. \\
\indent Assume there is a triangle in $G$. Coloring its vertices
with color $1$ and the remaining vertices with colors
$\{2,\ldots,k\}$ would contradict Claim 1.  If  $G$ has a $C_4$,
color its independent sets with colors $1$ and $2$, respectively,
and the remaining vertices with colors $\{3,\ldots,k\}$.  Under this coloring, any
rainbow $k$-vertex graph has $e(G)-2\Delta+1$ edges, another
contradiction
to Claim 1.\\

\noindent\textbf{Claim 4.} For any two nonadjacent vertices $u$ and $v$, $G-u-v\approx H$. \\
\indent Color $u,v$ and their common neighbor (which exists by Claim
2 and is unique by Claim 3) with color $1$ and rainbow color the
remaining vertices with colors $\{2,\ldots,k\}$.  Claim 1 implies
that $G-u-v$ must induce a copy of $H$.\\

\noindent\textbf{Claim 5.} $G$ is vertex-transitive. \\
\indent Let $v,v'\in V(G)$, let $x$ and $y$ be neighbors of $v$ and
let $x'$ and $y'$ be neighbors of $v'$.  There is an isomorphism
$\varphi : (G-x-y)\rightarrow (G-x'-y')$ that sends $v$ to $v'$,
since $v,v'$ are the unique degree $\Delta-2$ vertices  in the respective
copies of $H$.  To show that the map $\varphi$ can be extended to an
isomorphism of $G$ itself, we will verify, without loss of
generality, that $x'$ is adjacent to every vertex of $\varphi(N(x))$
and $y'$ is adjacent to every vertex of $\varphi(N(y))$.

First, note that $y$ and every vertex in $N(x)\setminus\{v\}$ have
exactly one common neighbor.  This neighbor, however, cannot be $v$
or $x$ because $G$ has no $K_3$.  Moreover, such neighbors are
different for each distinct member of $N(x)\setminus\{v\}$ because
otherwise that vertex and two of its neighbors would form a $C_4$
with $x$.

Therefore, there is an induced matching between $N(x)\setminus\{v\}$
and $N(y)\setminus\{v\}$.

Let $a\in N(x)\setminus\{v\}$.  Without loss of generality, suppose
$\varphi(a)\in N(x')$.  Let $b$ be any vertex in
$N(y)\setminus\{v\}$.  If $b\sim a$, then $\varphi(b)\in N(y')$
because $G$ has no $K_3$.  If $b\not\sim a$, then $a$ and $b$ have a
common neighbor in $G-x-y$, so $\varphi(b)\in N(y')$ because $G$ has
no $C_4$. Therefore, we can conclude that $\varphi(N(y))=N(y')$ and
symmetrically, $\varphi(N(x))=N(x')$, extending $\varphi$ to an
isomorphism of $G$.\\

\noindent\textbf{Claim 6.} If $G\arrows H$, $G$ contains no $K_3$
and $n\geq 6$, then $G\approx P$ or $G\approx \Theta$.\\
\indent Since $G$ is regular with diameter $2$ and girth $5$,
the Hoffman-Singleton theorem (Theorem \ref{HoffSingThm}) gives that
the only possibilities for $G$ are $C_5$, $P$, $\Theta$ or a
$(57,2)$-Moore graph, if it exists.  Since $n\geq 6$, $G\neq C_5$.
According to an unpublished proof due to Graham Higman, printed in
Section 3.7 of Cameron~\cite{C}, if a $(57,2)$-Moore graph exists,
then it cannot be vertex-transitive.  So, only $P$ and $\Theta$
remain.\\

\noindent\textbf{Claim 7.}   $\overline{P}$ and
$\overline{\Theta}$ are edge-transitive. \\
\indent Using the definition of the Petersen graph as a Kneser graph
(see Section 1.6 of Godsil and Royle, \cite{GR}), it is easy to see
that $\overline{P}$ is edge-transitive.

Now, we shall show that $\overline{\Theta}$ is edge-transitive.
The automorphism group of $\Theta$ is of
order $50\times 7!$ (see Brouwer, Cohen and Neumaier~\cite{BCN} or Hafner~\cite{H}) and the stabilizer of a vertex $w$ is $S_7$, the
symmetric group that permutes the neighbors of $w$. Take any pair of
nonadjacent vertices $x$ and $y$. Let $v$ be their common neighbor.
Any automorphism which fixes $\{x,y\}$ also fixes $v$. So, the
subgroup of automorphisms which fix $\{x,y\}$ is of order at most
$2\times 5!$.  By the orbit-stabilizer theorem (see Section 2.2 of
\cite{GR}), the orbit of a nonedge $\{x,y\}$ is of size at least
$\frac{50\times 7!}{2\times 5!}=1050$. The number of nonedges in
$\Theta$ is $1050$, hence $\overline{\Theta}$ is edge-transitive.\\
%  In fact, the stabilizer of an edge in
% $\overline{\Theta}$ is
% isomorphic to $\mathbb{Z}_2\times S_5$.\\

Putting all of the claims together, if the vertices of
$G\in\{P,\Theta\}$ are colored with $n-2$ colors, then either one
color class is of size $3$ or two color classes are each of size
$2$.  There is a pair of nonadjacent vertices that can be deleted so
that each vertex that remains is of a different color.  In the first
coloring, this is because $G$ has no $K_3$; in the second, because
$G$ has no $C_4$.

Since $\overline{G}$ is edge-transitive, the deletion of any
nonadjacent vertices produces a graph isomorphic to  $P'$ or
$\Theta'$, respectively.  By adding the complementary cases, the
Lemma follows. $\Box$\\~\\

\subsection{Proof of Lemma \ref{tool}}~\\
Let $n=k+2$, $\Delta-\delta=3$, $\Delta(G)-\delta(G)=1$. Let $S$ be
a set of vertices that induces $H$ and let
$\Delta=\deg(G[S],y_k)\geq\cdots\geq\deg(G[S],y_1)=\delta$.  By
Lemma \ref{neighborhood-into-S}, part \ref{DELdel}, there are two
possibilities, {CASE 1} and {CASE 2}:\\

\noindent\textbf{CASE 1.} There is a $v_0\in V(G)\setminus S$ such that $G\left[\{v_0\}\cup S\setminus\{y_k\}\right]\approx H$. \\
\indent Let $\{v_1\}=V(G)\setminus (S\cup\{v_0\})$.  Since
$\left|N(v_0)\cap (S\setminus\{y_k\})\right|=\Delta$ and
$\Delta=\Delta(G)$, it is the case that $v_1\not\sim v_0$.  Color
$\{v_0\}\cup S\setminus\{y_1,y_2\}$ with colors $\{1,\ldots,k-1\}$
and color $\{y_1,y_2,v_1\}$ with color $k$.  In the rainbow copy of
$H$, $v_1$ must be chosen; otherwise the resulting rainbow subgraph
has at least $e(H)-(\delta+1)+(\Delta-1)>e(H)$ edges.

So, $U:=V(G)\setminus\{y_1,y_2\}$ induces $H$.  Count the number of edges in the subgraph induced by $U$:
\begin{eqnarray}
   e(G[U])=e(H) & \geq & e(H)-|S\cap N(y_2)|-|S\cap N(y_1)|
   +(\deg(G,v_0)-2)+(\deg(G,v_1)-2) \nonumber \\
   & \geq & e(H)-(\Delta-2)-(\Delta-3)+\deg(G,v_0)
   +\deg(G,v_1)-4 \label{eq:DEGS1} \\
   &=& e(H)-2\Delta+5+\Delta+\deg(G,v_1)-4 \nonumber \\
   &=& e(H)-\Delta+1+\deg(G,v_1). \nonumber
\end{eqnarray}
So, $\deg(G,v_1)\leq\Delta-1$ but this must occur with equality because $\delta(G)=\Delta-1$.  Since $\deg(G, v_0)=\Delta$, $\deg(G,v_1)=\Delta-1$, $v_0\not \sim v_1$, we have that
\begin{equation}\label{e1}
e(H)=e(G)-2\Delta +1.
\end{equation}

If there are three vertices of degree $\Delta-1$ in $G$, then color
each of them with color $k$ and rainbow color the remaining vertices
with colors $\{1,\ldots,k-1\}$.  Any rainbow colored graph on $k$
vertices under this coloring would have at least $e(G)-2\Delta+2$
edges, a contradiction to (\ref{e1}).  Thus, there are at most two
vertices of degree $\Delta-1$ in $G$.

Because $\deg(G,v_1)=\Delta-1$, equality holds in (\ref{eq:DEGS1})
and, in particular, $|S\cap N(y_1)|=\Delta-3$. Since
$\delta(G)=\Delta-1$, we have that $y_1\sim\{v_0,v_1\}$ and
$\deg(G,y_1)=\Delta-1$. Thus, $v_1$ and $y_1$ are the only vertices
of degree $\Delta-1$ in $G$.  The vertex $y_2$ has degree $\Delta$
in $G$ but the equality in (\ref{eq:DEGS1}) gives that $y_1\not\sim
y_2$ and $|S\cap N(y_2)|=\Delta-2$.  So, $y_2\sim\{v_0,v_1\}$.
Putting all of this information together, we arrive at the fact that
$\{v_0,y_2,v_1,y_1\}$ forms an induced copy of $C_4$ with
$v_0\not\sim v_1$ and $y_1\not\sim y_2$.

If we rainbow color $S\setminus\{y_2,y_1\}$ with colors
$\{1,\ldots,k-2\}$ and color the vertices in $\{y_2,y_1\}$ with
color $k-1$ and $\{v_0,v_1\}$ with color $k$, then the only
possibility for a rainbow graph with $e(G)-2\Delta+1$ edges in this coloring is $U':=V(G)\setminus\{v_0,y_2\}$.  Since $G\arrows H$, the graph induced by $U'$ must be isomorphic to $H$. Let $w$ be a vertex of degree $\Delta-3$ in $U'$, then $\deg(G,w)=\Delta-1$. This forces $w$ to be either $v_1$ or $y_1$. However, both $v_1$ and $y_1$ have exactly one neighbor among $\{v_0,y_2\}$, giving that $\deg(G[U'],v_1)=\deg(G[U'],y_1)=\Delta-2$, a contradiction.\\

\noindent\textbf{CASE 2.} There is a $v_1\in V(G)\setminus S$ such that $G\left[\{v_1\}\cup S\setminus\{y_1\}\right]\approx H$. \\
\indent Consider ${\overline G}$ and ${\overline H}$. We have that
${\overline G}\arrows {\overline H}$ and ${\overline G}$ is bounded
by ${\overline H}$.  Observe that $y_1$ in an induced copy of $H$ in
$G$ corresponds to $y_k$ in the same set of vertices, which induce a
copy of ${\overline H}$ in ${\overline G}$. Thus we have {CASE 1}
for ${\overline G}$ and ${\overline H}$, resulting
in contradiction. $\Box$\\~\\

\subsection{Proof of Lemma \ref{unbounded-unique-delta}}~\\
Recall that $\delta=1$.  Suppose $v$ is a vertex such that
$\deg(G,v)=\delta(G)<n-k+\delta$. Then $\codeg(G,v) \geq n-(n-k+\delta) = k-\delta=k-1$.
Color
$k-1$ non-neighbors of $v$
with distinct colors,  color the rest of the graph with the
remaining colors. Let $S$ induce a rainbow copy of $H$ in $G$ in
this coloring.

Suppose $v\in S$.  Since $k-1$ non-neighbors of $v$ must be in
$S$, we see that $H$ would contain a vertex (namely, $v$) with at
least $k-1$ non-neighbors, and thus having degree $0$, a contradiction.  As a result,
$v\not\in S$.

Label the vertices of $S$ so that
$\Delta=\deg(G[S],y_k)\geq\cdots\geq\deg(G[S],y_1)=\delta$.  Let
$U=\{v\}\cup S\setminus\{y_k,y_{k-1}\}$.  Using Lemma \ref{deck}, we
have that $e(G[U])\geq e(H)-\Delta$. On the other hand, we know that
$v$ has at least $k-1$ non-neighbors in $S$, and so it has at
most one neighbor in $S$.

As a result, $e(G[U])\leq e(H)-\Delta-(\Delta-1)+1+1$ and
$e(G[U])\geq e(H)-\Delta$, so $\Delta\leq 3$.  If equality
occurs then, by Lemma \ref{neighborhood-into-S}, part \ref{DEL},
$H_3=\{y_k\}$ and $y_k\sim H_{2}$.  Moreover, the single
member of $N(v)\cap S$ is neither $y_k$ nor $y_{k-1}$. Furthermore,
equality also implies that it is not possible to find a $y_{k-1}$
adjacent to $v$.  Hence, $v\not\sim H_3\cup H_2$. $\Box$
\\~\\

%Suppose $\deg(G,v)>\Delta=k-2$.  Rainbow color $\Delta+1$
%   neighbors of $v$ with colors $\{1,2,\ldots,\Delta+1\}$.
%   Color the rest with color $k$.  If $S$ induces  a rainbow
%   copy of $H$, then $v\not\in S$.  However
%   $|N(v)\cap S|=\Delta+1$.  Let $T=\{v\}\cup
%   S\setminus\{y_1,y_2\}$.  Then, using  Lemma \ref{deck}, we have
%   that $e(G[T]]) \leq e(H)-\delta$. On the other hand, by
%   considering the number of edges incident to $y_1, y_2$, we have
%   that $e(G[T])\geq e(H)-\delta-(\delta+1)+(\Delta-1) .$
%   Therefore, $\Delta-\delta\leq 2$ and equality only occurs if
%   $y_1$ is the unique vertex of degree $\delta$, is not adjacent
%   to any vertex of degree $\delta+1$ and $v$ is adjacent to both
%   $y_1$ and all of $H_{\delta+1}$ (since $y_2$ can be chose
%   arbitrarily from $H_{\delta+1}$.  By using the fact that
%   $G\arrows H$ iff $\overline{G}\arrows\overline{H}$, we arrive
%   at the second statement of the Lemma.\\~\\

\subsection{Proof of Lemma \ref{matching-isolated}}~\\
Let $\Delta(H)=1$ and
$H\not\in\left\{K_2,\overline{K_2},\overline{P_3}\right\}$.  Hence,
$k\geq 4$ and $\Delta(H)\leq k-3$.  Let $G\arrows H$, where $G$ is a
graph on $n\geq k+2$ vertices.  By Lemma \ref{bounded},
$\Delta(G)=1$.  We have the following cases:\\

\noindent\textbf{CASE 1.} $G$ has more isolated vertices than $H$.\\
\indent Color the isolated vertices of $G$ with as many colors as
possible using at most $k-1$ colors, color the vertices of degree
$1$ in $G$ with the remaining colors. It is clear that any rainbow
subgraph on $k$ vertices will have more isolated vertices than
$H$.\\

\noindent\textbf{CASE 2.} $G$ has more edges than $H$ and $H$ has at
least three isolated vertices.\\
\indent In this case, the number of vertices of degree $1$ in $H$ is
at most $k-3$.  Color vertices in as many edges of $G$ as possible
with distinct colors, using at most $k-1$ colors, and color the rest
of the graph with the remaining colors.  Any rainbow subgraph on $k$
vertices will contain more edges than $H$.\\

\noindent\textbf{CASE 3.} $G$ has more edges than $H$ and $H$ has at
most two isolated vertices.\\
\indent Assume first that $n\geq k+3$, then color each of the
isolated vertices of $G$ with distinct colors and, for as many edges
as possible, color the endvertices with the same color, a different
color on each edge. We see that any rainbow $k$-vertex subgraph has
at least $3$ isolated vertices.

Assume now that $n=k+2$. If $G$ has at least one isolated vertex,
then color the endpoints of one edge with color $1$, color the
endpoints of another edge with color $2$, rainbow color the rest.
Then, $H$ must have three isolated vertices, a contradiction. Thus,
$G$ is a matching. If, as before, we color the vertices of one edge
with color $1$, then the vertices in another edge with color $2$ and
the rest with remaining colors, then $H\approx M_{k/2+1}'$.

To complete the proof, observe that any coloring of $G\approx M_{k/2+1}$ with $k$ colors gives such a
rainbow $H\approx M'_{k/2+1}$. $\Box$\\~\\

\subsection{Proof of Lemma \ref{Delta2}}~\\
\indent Let $\Delta=2$, $H\not\in\{\Lambda,\overline{S_4}\}$.\\

\noindent\textbf{CASE 1.}  $H$ has two nonadjacent vertices of degree $2$.\\
\indent If $k\leq 4$, the only possibility is $H\approx C_4$.
Corollary 2 from the paper \cite{A} gives that $f(H)=k$ for any
regular graph $H$.

Hence, we may assume $k\geq 5$ and $\Delta=2\leq k-3$.  By Lemma
\ref{bounded} part \ref{lem:deg}, $\Delta(G)=\Delta=2$.
We shall show that $G$ is a disjoint union of cycles each of length at least $5$. \\

\noindent
{\bf Claim 1.} Let $S\subset V(G)$ be any vertex set that induces a copy of $H$.
Then $V(G)\setminus S$ is an independent set.\\

For all $v\in V(G)\setminus S$  Lemma
\ref{neighborhood-into-S} part \ref{DEL} gives that  $|N(v)\cap
(S\setminus\{y_k,y_{k-1}\})|=\Delta=2$.  Each vertex in
$V(G)\setminus S$ sends $2$ edges into $S$, and thus no edges into $V(G)\setminus S$. So $V(G)\setminus S$ induces an
empty graph. \\

\noindent
{\bf Claim 2.} $G$ is a union of cycles.\\

Assume that there is a vertex $v$ in $G$ such that  $\deg(G,v)<2$,
consider any two adjacent vertices $w,w'$.  Color $v,w,w'$ with one
color and the rest of the graph with the remaining $k-1$ colors.
Let $U\subseteq V(G)$  induce a rainbow copy of $H$ under this coloring.
Then $U$ contains one vertex from $\{v,w,w'\}$, moreover $v\in U$, since its degree is less than $2$.
On the other hand, both $w$ and $w'$ cannot be in $V(G)\setminus U$ since
they are adjacent, contradicting the fact that $V(G)\setminus U$
is an independent set.  Thus there is no vertex of degree less than $2$ in $G$.
Therefore $G$ is $2$-regular, and Claim 2 follows.\\

\noindent
{\bf Claim 3.} $G$ has no $K_3$ and no $C_4$.\\

Assume that $G$ has a triangle.  Color its three vertices with the same
color and color the rest of the vertices with new colors.  In this
coloring, there is a pair of adjacent vertices outside of a rainbow
copy of $H$, a contradiction. If $G$ has a $C_4$, then color each of the
two nonadjacent vertices of a copy of $C_4$ with color $1$, color
each of the two other vertices of $C_4$ with color $2$ and the color
the rest of the vertices of $G$ with the remaining $k-2$ colors
(here $n\geq k+2$ is necessary).  This gives two adjacent vertices
outside of a rainbow copy of $H$, a contradiction. \\

If $n\geq k+3$, then color four consecutive vertices on one cycle
with one color, and color the rest of the vertices arbitrarily with
the remaining $k-1$ colors.  As a result, there is an edge outside
of the rainbow copy of $H$, a contradiction.

Thus we  may assume that $n=k+2$. Color three consecutive vertices on one of the cycles of
$G$ with color $1$ and rainbow color the rest of the vertices with
the remaining colors. Since there is no edge outside of a copy of
$H$ in $G$, we must pick the middle vertex of color $1$ in a rainbow
copy of $H$. Thus, $H$ is a disjoint union of an isolated vertex, a
path and perhaps some cycles. If $G$ has at least two cycles, then
color two vertices in one cycle with color $1$, color two vertices
in another cycle with color $2$ and rainbow color the rest of the
vertices with the remaining $k-2$ colors.  Under this coloring, $H$
has no isolated vertex, a contradiction.  The only case that remains
is that $G$ is a single cycle and $H=P_{k-1}+K_1$. Since $H$ has two
nonadjacent vertices of degree two, $k-1\geq 5$, thus $n\geq 8$.
Let  $v,v'$ and $u,u'$ be two pairs of consecutive vertices  of $G$
such that $G-u-u'-v-v'$ consists of two paths, each of length at least
$2$.  Color $v,v'$ with color $1$ and $u,u'$  with color $2$, rainbow color the rest of the vertices with remaining colors.
Under this coloring, $H$ has no isolated vertices, a contradiction.\\

\noindent\textbf{CASE 2.} All degree $2$ vertices of $H$ are
adjacent.\\
\indent Recall that $n\geq k+2$. We have that $H$ has one component
$L\in \{P_3,K_3,P_4\}$, and all other components are isolated edges
and vertices. We may assume that $k\geq 4$ since
$H\not\in\mathcal{F}_{\infty}$. If $k=4$ then
$H\in\{P_4,\Lambda,\overline{S_4}\}$. Since
$H\not\in\mathcal{F}_{\infty}$, we see that $H\approx P_4$. Since
$n\geq k+2=6$, Ramsey's theorem gives that $G$ contains either $K_3$
or $\overline{K_3}$, contradicting the Deck Lemma (Lemma
\ref{deck}). We can assume that $k\geq 5$,  giving that
$\Delta=2\leq k-3$ and $\Delta(G)=2$ by Lemma \ref{bounded}. The
Deck Lemma implies that each connected subgraph  of $G$ on $|V(L)|$
vertices is isomorphic to $L$.

Assume first that $k\geq 7$.  There is only one component of $G$
with at least three vertices (call such a component large) and this
large component is either $P_3$, $K_3$ or $P_4$. Indeed, otherwise
one can find two nonadjacent vertices of degree $2$ in $G$;
considering these and their neighbors will contradict the Deck
Lemma.  Since $n\geq k+2$, we can color the vertices of the large
component of $G$ with two colors and color the remaining vertices
arbitrarily.  Under this coloring, any rainbow subgraph has
components with at most two vertices, a contradiction.

Thus, $k=5$ or $6$ and $L\in\{P_3,K_3,P_4\}$.

Let $k\in\{5,6\}$ and $L\in \{P_3, K_3\}$.  Then, $G$  has no
component on more than three vertices, and each large component of
$G$ is isomorphic to $L$. If $G$ has only one component isomorphic
to $L$, then color it with two colors, and color the rest
arbitrarily, resulting in contradiction. If $G$ has at least two
such components, color the vertices in two copies of $L$ with two
colors each and color the rest of $V(G)$ with the remaining $1\leq
k-4\leq 2$ colors.  Under this coloring, no rainbow subgraph has a
component with more than $2$ vertices.

Let $k=6$ and $L\approx P_4$.  Then, $H$ is either $P_4+K_2$ or
$P_4+2K_1$. Any component of $G$ on at least three vertices must be
$P_4$ since otherwise there is a subgraph $P_5$ of $G$ which
contains two vertices of degree $2$, nonadjacent in $G$, such that
these two vertices and their neighbors span at most $5$ vertices,
which contradicts the Deck Lemma.  Color the vertices of $P_4$ in
$G$ with three colors and color the rest of the graph with the
remaining three colors.  Under this coloring, no rainbow subgraph has a $P_4$.

Let $k=5$ and $L\approx P_4$. Then $H\approx P_4+K_1$. In
particular, we have that $H$ has no induced $2K_2$. Thus, $G$ has
only one nontrivial (with at least one edge) component and this
component is either $P_4$ or $C_5$. Since $n\geq k+2$, we can color
this component with three colors and color the rest of the vertices
arbitrarily with the remaining $2$ colors, arriving at a
contradiction. $\Box$\\~\\

\subsection{Proof of Lemma \ref{03}}
Recall that $\Delta=\Delta(H)=3$, $\delta=\delta(H)\leq 1$ and
$\delta(G)<n-k+\delta$. \\

\noindent\textbf{Claim 1.} There exists a set $S\subset V(G)$ and vertices $y_1\in S$ and $v_1\in V(G)\setminus S$ such that $G[S]\approx H$, $y_1$ is a minimum-degree vertex in $G[S]$ and $N(v_1)\cap S=\{y_1\}$. \\
\indent If $\delta=1$, then this follows directly from Lemma
\ref{unbounded-unique-delta}.

If $\delta=0$, then suppose the claim is false.  Lemma
\ref{neighborhood-into-S}, part \ref{DELdel} gives that there is a
$v_0\in V(G)\setminus S$ such that
$S_0:=\{v_0\}\cup\left(S\setminus\{y_k\}\right)$ and $G[S_0]\approx
H$. In this case, $\left|N(v_0)\cap\left(S\setminus\{y_1,y_2\}\right)\right|\leq
\delta+1=1$, so $v_0\sim\{y_1,y_2\}$ and Lemma
\ref{neighborhood-into-S}, part \ref{del} also gives that $y_1$
is the unique isolated vertex in $G[S]$.  We have the freedom to
choose $y_2$ to be any degree-one vertex in $G[S]$, hence $v_0$ is
adjacent to every vertex of degree at most one in $G[S]$.

As a result, $G[S_0]$ has no isolated vertices, a contradiction to
the claim that $G[S_0]\approx H$.  This proves Claim 1. \\

\noindent\textbf{Claim 2.} $|H_3|=1$, $H_3\sim H_2$. \\
\indent If $\delta=1$, this follows directly from Lemma
\ref{unbounded-unique-delta}.  If $\delta=0$, this comes from Claim
1 and Lemma \ref{neighborhood-into-S}, part \ref{DEL}.  This
proves Claim 2. \\

\noindent\textbf{CASE 1.} $\delta=0$. \\
\indent By Lemma \ref{neighborhood-into-S}, part \ref{del}, the
graph $G[S]$ has a unique isolated vertex, $y_1$.  Since $|H_3|=1$,
the component containing $y_k$ has an odd number of degree-one
vertices.

If the component of $G[S]$ containing $y_k$ has 3 degree-one
vertices, call two of them $y_2$ and $y_3$.  Let
$S'=\{v_1\}\cup\left(S\setminus\{y_2,y_3\}\right)$, where $v_1$ as in Claim 1. The graph
$G[S']$ has no isolated vertices.  Since $|S'|=k-1$ and
$e(S')=e(H)-1$, the Deck Lemma implies that $S'$
is obtained by deleting a degree-one vertex from $H$, which would
yield at least one isolated vertex, a contradiction.

Therefore, we may assume that the component of $G[S]$ containing
$y_k$ has exactly 1 degree-one vertex.  Since all degree-two
vertices must be adjacent to $y_k$, the vertex set
$\{y_k,y_{k-1},y_{k-2}\}$, for two degree-two vertices $y_{k-1}$ and
$y_{k-2}$, induces a triangle. Let
$S''=\{v_1\}\cup\left(S\setminus\{y_{k-1},y_{k-2}\}\right)$. The
graph $G[S'']$ has no isolated vertices.  Since $|S''|=k-1$ and
$e(S'')=e(H)-2$, the Deck Lemma implies that $S''$ is obtained by
deleting a degree-two vertex from $H$, which would yield at least
one isolated vertex, a contradiction.

Hence, there is no graph in CASE 1. \\

\noindent\textbf{CASE 2.} $\delta=1$. \\
\indent If $y_1\not\sim y_k$, then let
$S'=\{v_1\}\cup\left(S\setminus\{y_k,y_1\}\right)$, where $v_1$ as in Claim 1. Since
$|S'|=k-1$ and $e(S')=e(H)-4$, the Deck Lemma is contradicted.
Therefore, we may assume that $y_1\sim y_k$.

If $H$ is not connected, then since every degree-two vertex in
$G[S]$ is adjacent to $y_k$, every connected component of $G[S]$ not
containing $y_k$ must be an isolated edge.  Let $\{y_2,y_3\}$ be a
component of $G[S]$ not containing $y_k$.  If
$S''=\{v_1\}\cup\left(S\setminus\{y_2,y_3\}\right)$, then
$e(S'')=e(H)$, contradicting the Deck Lemma.  Therefore, we may also
assume that $H$ is connected.

Since $H_3\sim H_2$, $y_1\sim y_k$ and $H$ is connected, there are
only three possibilities for $H$: one for each of $k=4,5,6$.  If
$k\leq 4$, then $H\approx\overline{\Lambda}$.  If $k\in\{5,6\}$ then
Figure \ref{fig:3new} gives these graphs and the possible ways for
$v_1$ to be adjacent to $S$.

\begin{figure}[ht]
\begin{center}
\epsfig{file=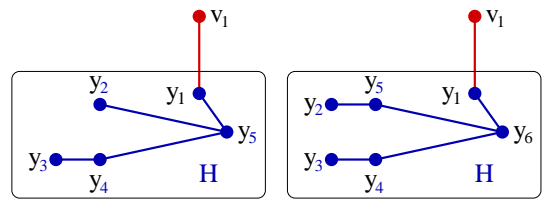, height=1in}
\end{center}
\caption{Small graphs, $\delta=1, \Delta=3$ in Lemma
\ref{03}}\label{fig:3new}
\end{figure}

In the case where $k=5$, the vertex set $\{v_1,y_1,y_3,y_4\}$
induces the graph $2K_2$, which is not in $\textrm{deck}(H)$.  In
the case where $k=6$, the vertex set $\{v_1,y_1,y_2,y_3,y_6\}$
induces the graph $P_3+2K_1$, which again is not in
$\textrm{deck}(H)$.

Hence, the only graph in CASE 2 is $H\approx\overline{\Lambda}$,
which must be excluded because
$\overline{\Lambda}\in\mathcal{F}_{\infty}$. $\Box$\\~\\

\section{Concluding remarks}

{\bf Open question: $H\approx\Lambda, H\approx \overline{\Lambda}$}~\\
It is still an open problem to determine ${\mathcal Arrow}(\Lambda)$ and ${\mathcal Arrow}(\overline{\Lambda})$.
We see by Lemma \ref{k+1} that $C_4+K_1\arrows\Lambda$;  it is shown in \cite{A}, that for each $t\geq 1$ there is a graph $G$
on $7t$ vertices such that $G\arrows\Lambda$.
A case analysis, which we neglect to include in this paper, gives
that for any $G$ of order $6$,  $G\notarrows\Lambda$.

Thus, even the weaker problem of determining $\{n: |V(G)|=n \mbox { and } G\arrows \Lambda \}$
is still open.  \\

{\bf Generalizing the problem}~\\
A natural generalization of this problem is as follows: Let
$\mathcal{H}$ be a set of $k$-vertex graphs and define
$G\arrows\mathcal{H}$ so that if $V(G)$ is colored with $k$ colors,
then there is an $H\in\mathcal{H}$ such that $G$ contains a rainbow induced subgraph isomorphic to $H$. Determine ${\mathcal Arrow}(\mathcal{H})= \{ G: G\arrows \mathcal{H}\}$,
for interesting sets $\mathcal {H}$ of graphs. \\

%\subsection{Moore graph}~\\
%It is still not known if any $(57,2)$-Moore graph exists.\\~\\

{\bf Acknowledgements}\\
We thank an anonymous referee for helpful comments.  We also profoundly acknowledge and  thank Chris Godsil for
helpful remarks and aiding with the proof of Lemma \ref{regular}.

\end{document}